\documentclass[twoside, 11pt]{article}
\usepackage{authblk}
\usepackage{amsfonts, amsmath, amssymb, amsthm, mathtools, relsize}
\usepackage{hyperref}
\usepackage{graphicx}
\usepackage[normalem]{ulem} 
\usepackage{fullpage}
\usepackage{cite}
\usepackage{enumerate}
\usepackage{graphicx}
\usepackage{color}

\usepackage{tikz}
\usetikzlibrary{automata, positioning}
\usetikzlibrary{shapes.geometric, arrows, shapes.misc}

\definecolor{nb}{rgb}{.6, .176, 1}
\definecolor{sienna}{rgb}{1, 0, 0}
\definecolor{darkgreen}{rgb}{0, .5, 0}

 \newtheorem{theorem}{Theorem}[section]
 
 \newtheorem{corollary}[theorem]{Corollary}
 
 \newtheorem{lemma}[theorem]{Lemma}
 \newtheorem{example}[theorem]{Example}
 
 \newtheorem{remark}{Remark}

\newcommand{\hZ}{\mbox{$\hat{Z}$}}

\newcommand{\GG}{\mbox{${\mathcal G}$}}
\newcommand{\FF}{\mbox{${\mathcal F}$}}

\newcommand{\LL}{\mbox{${\mathcal L}$}}
\newcommand{\OO}{\mbox{${\mathcal O}$}}

\newcommand{\bD}{\mbox{${\mathbf D}$}}
\newcommand{\bT}{\mbox{${\mathbf T}$}}

\newcommand{\IR}{\mbox{${\mathbb R}$}}
\newcommand{\IZ}{\mbox{${\mathbb Z}$}}
\newcommand{\IE}{\mbox{${\mathbb E}$}}

\newcommand{\I}{\mbox{${\mathbb I}$}}

\newcommand{\bone}{\mathbf{1}}
\newcommand{\bz}{\mathbf{0}}
\newcommand{\bw}{\mathbf{w}}

\newcommand{\diin}{d_i^{\mathrm{in}}}
\newcommand{\djin}{d_j^{\mathrm{in}} }
\newcommand{\diout}{d_i^{\mathrm{out}}}

\newcommand{\tA}{\tilde{A}}

\newcommand{\Var}{\mathop{\mathrm{Var}}}

\begin{document}

 \title{{\bf Interacting Urns on a Finite Directed Graph}}
\author[*]{Gursharn Kaur} 
\author[**]{Neeraja Sahasrabudhe}

\affil[*]{National University of Singapore, Singapore}
\affil[**]{Indian Institute of Science Education and Research, Mohali, India}
\date{14 March, 2021}

	 \maketitle

\begin{abstract}
We introduce a general two colour interacting urn model on a finite directed graph, where each urn at a node,
 reinforces all the urns in its out-neighbours according to a fixed, non-negative and balanced  reinforcement matrix. 
 We show that the fraction of balls of either colour converges almost surely to a deterministic limit if either the reinforcement is not of P\'olya type or if the graph is such that every vertex with non-zero in-degree can be reached from some vertex with zero in-degree.  We also obtain joint central limit theorems, with appropriate scaling, around the vector of limiting proportion.  Further, in the remaining case when there are no vertices with zero in-degree and the reinforcement is of P\'olya type, we restrict our analysis to a regular graph and  show that the fraction of balls of either colour  converges almost surely to a finite random limit, which is the same across all the urns. 
\end{abstract}

\


\


\section{Introduction} \label{Sec:intro}
Systems with multiple components that evolve randomly through self-reinforcement or reinforcement via interactions with other components of the system have been of great interest for a long time. Interacting urn models are often used to analyse such systems. Recently, there has been a lot of activity in the area of interacting urns \cite{SAurns, Paolonetworkurns, Irene-Paolo2014, Paolo-Idda2013, graph, NS}. 
In simplest terms, an urn model or an urn process refers to a discrete time random process that involves updating the configuration of an urn consisting of balls of different colours, according to some reinforcement rule. Here by the configuration of an urn at time $t$, we mean a vector where the $i^{\mathrm{th}} $ element represents the number of balls of  colour $i$ in the urn at time $t$.  The reinforcement process is assumed to be Markovian, that is, the reinforcement at any time depends only on the present urn configuration. The most common way of self-reinforcement is to select a ball uniformly at random from the urn at time $t$ and then depending on the colour of the drawn ball,  a certain number of balls of some colours  are added or removed from the urn. Such an urn model can be fully described by the initial configuration of the  urn and the associated replacement matrix, whose $(i,j)^{\mathrm{th}}$ element is the number of balls of colour $j$ added  or removed from the urn when the ball drawn is of colour $i$. 

Traditionally, the study of urn models is classified based on the type of replacement matrix. For instance, the replacement matrix  for the classical {P}\'{o}lya urn model is the identity matrix. That is, at every time-step a ball is drawn and is replaced in the urn along with another ball of the same colour.  The asymptotic properties of the P\'olya urn model have been studied extensively \cite{Mahmoud2009}. The most well-known result for the {P}\'{o}lya urn model is that the fraction of balls of each colour converges  almost surely to a  random limit, which is distributed according to beta distribution with parameters depending on the initial configuration of the urn.  An immediate generalisation of the two-colour {P}\'{o}lya urn model  is the Friedman urn model \cite{Friedman}, where the chosen ball is replaced with $\alpha\geq 0$ balls of the same colour and $\beta >0 $ balls of the other colour. In this case, the fraction of balls of either colour approaches the deterministic limit $1/2$ with probability $1$. Several other generalisations of the P\'olya urn model have been studied and we refer the reader to \cite{Mahmoud2009} for details.

\subsection{General two-colour interacting urn model} \label{sec1.1}

In this paper, we study urn processes involving more than one urn.
We define a \textit{general two-colour interacting urn model}  as a random process involving $N$ urns such that, the reinforcement in each urn depends on all the urns or on a non-trivial subset of the given set of $N$ urns. 
More precisely, suppose there are $N$ urns with configurations $(W^t_i, B^t_i) $, where $W^t_i$ and $B^t_i$ denote the number of white balls and black balls respectively, at time $t\geq 0$ in the $i^{\mathrm{th}}$ urn, for every $i\in [N] \coloneqq \{1, \ldots, N \}$.  Let $Z_i^t = \dfrac{W_i^t}{W_i^t+B_i^t}$ be the proportion of white balls in the $i^{\mathrm{th}}$ urn  at time $t$ and $\FF_t = \sigma\left( Z_i^s: 0\leq s\leq t,  i\in [N] \right)$.
Define $ (I_{i,W}^{t+1}, I_{i,B}^{t+1}) \coloneqq   (W^{t+1}_i, B^{t+1}_i) -(W^t_i, B^t_i)$ to be the reinforcement in the $i^{\mathrm{th}}$ urn at time $t+1$. 
We consider  non-negative and  finite reinforcement, that is $0\leq I_{i,W}^t, I_{i,B}^t <\infty $  for every $i \in [N]$ and $t\geq 0$, such that $\{(I_{i,W}^t, I_{i,B}^t )\}_{i\in [N]}$ are conditionally independent given $\FF_{t-1}$ for every $t\geq 1$.
If the evolution of the $i^{\mathrm{th}}$ urn depends on urns $\{i_1, \dots,i_{k_i} \} \subseteq [N]$,  then  the distribution of $(I^t_{i,W}, I^t_{i,B})$, conditioned on $\FF_{t-1}$ is determined by $Z^{t-1}_{i_1}, \dots, Z^{t-1}_{i_{k_i}}$.
 We call the set $\{i_1, \dots,i_{k_i} \}$ the \emph{dependency set} of the $i^{\mathrm{th}}$ urn.   In particular, if we consider a graph based general two-colour interacting urn model, a natural choice for the dependency set of an urn is the collection of urns in its neighbourhood.

\sloppy Several special cases of the general two-colour interacting urn model have been studied recently. In  \cite{Irene-Paolo2014, Paolo-Idda2013} and \cite{NS}, interacting urn models were studied with {P}\'{o}lya  and Friedman urns respectively.  More precisely, in \cite{Irene-Paolo2014, Paolo-Idda2013} the authors take  $ P\left((I^t_{i,W}, I^t_{i,B}) = (1,0)|\FF_{t-1}\right)  = p Z_i^{t-1}+(1-p)\frac{1}{N} \sum \limits_{j=1}^N Z_j^{t-1} $ and $P\left((I^t_{i,W}, I^t_{i,B}) = (0,1)|\FF_{t-1}\right)  =1- P\left((I^t_{i,W}, I^t_{i,B}) = (1,0)|\FF_{t-1}\right)$,  for every $i\in [N]$ and some fixed $p\in [0,1]$. Thus, the dependency set of every urn is the entire set of urns.  In other words, the underlying network of interactions is a  complete undirected graph on $N$ vertices with self-loops at every vertex. Graph based interactions in urn processes have been studied before in \cite{graph, Remco2016, Crimaldi2020}.
 
In this paper, we consider $N$ interacting two-colour urns, which are placed at the vertices of a finite directed graph. A ball is drawn independently and uniformly at random, simultaneously from all the urns. Each urn then reinforces all the urns in its out-neighbourhood, according to a reinforcement  matrix $R$. 
More precisely, if  $R =  \begin{pmatrix}  \alpha & \delta \\ \gamma & \beta \end{pmatrix}$ is a reinforcement matrix, then if a white ball is drawn from an urn, all its out-neighbours are reinforced with $\alpha$ white and $\delta$ black balls; and if the ball drawn is black, the out-neighbours are reinforced with $\gamma$ white and $\beta$ black balls. 
 We restrict our discussion to non-negative reinforcement matrix which is \emph{balanced},  that is $\alpha + \delta = \gamma + \beta$.  

Note that  under the above dynamics, the urns at the isolated vertices in the graph are neither reinforced, nor do they reinforce any other urn. We therefore assume, without any loss of generality, that there are no isolated vertices in the graph. Further, since the urns at vertices with zero in-degree are never reinforced, their initial configuration is preserved.

\subsection{Summary of results}
We state our main results in Section \ref{Sec:MainResults}. To study the asymptotic properties of $Z_i^t$ for the proposed model with reinforcement matrix $R$,  we divide the discussion into two parts based on the type of reinforcement matrix namely, {P}\'{o}lya type (when $\delta=\gamma=0$) and non-{P}\'{o}lya type (when $\gamma+\delta >0$ and $\alpha+\beta>0$). For non-{P}\'{o}lya type reinforcement, we show that the fraction of white balls in the urns at vertices with non-zero in-degree converge to a deterministic limit almost surely. 
 Further, under certain conditions on the graph and the initial configuration of the urns we observe \emph{synchronisation}, that is,  the almost sure limit is the same across all urns at vertices with non-zero in-degree. We also obtain  fluctuation theorems around the limiting vector.
 
For P\'olya type reinforcement and the case when $\alpha+\beta=0$, we obtain results depending on the presence of the vertices with zero in-degree in the graph.  When there are vertices with zero in-degree such that every vertex with non-zero in degree in the graph can be reached from at least one vertex with zero in-degree,  we show that the fraction of white balls in every  urn converges to a deterministic limit almost surely. 
When there are no vertices with zero in-degree, we limit our discussion of P\'olya type reinforcement to a regular directed graph, that is when the in-degree and the out-degree of every vertex equal a constant $d$. 
In this case  we show that the fraction of balls of white colour in every urn converges to the same  random limit almost surely and  we also obtain $\LL^2$-rates of synchronisation. However, the limiting distribution in this case is not known. 
On this note, we remark that while the reinforcement matrix is of {P}\'{o}lya type, at any given time-step an urn can be reinforced with  balls of both colours, depending upon the balls drawn at that time from its in-neighbours.  Thus, in our model the  random process of  a single urn is nowhere similar to the classical {P}\'{o}lya process and we expect the problem of finding the limiting distribution to be fairly challenging.

We use  stochastic approximation method and the martingale method to study the interacting urn model proposed in this paper. The stochastic approximation  method has been used before (see \cite{Zhang2016, SAurns}) to obtain several interesting results for random processes with self or interactive reinforcement. We refer the reader to \cite{borkar} for  details on stochastic approximation theory. For a  larger discussion on various techniques used to study random processes with reinforcement (including  urn models) we refer the reader to  \cite{survey}.

\subsection{Outline of the paper}
The rest of the paper is organised as follows: in Section~\ref{Sec:prelim} and \ref{Sec:SA}, we describe the model in detail and the corresponding stochastic approximation scheme respectively. In Section~\ref{Sec:MainResults}, we state the main results. In Section~\ref{Sec:proofs}, we present the proofs of all the results stated in Section~\ref{Sec:MainResults}.  In Section~\ref{Sec:Application} we discuss an application of our results to study an opinion dynamics model on finite directed networks.

\subsection{Notation}
Throughout the paper, we use the following notation: for a sequence of random variables $\{ X_t \}_{t \geq 0}$ and a random variable $X$, $X_t \xrightarrow{d} X$ means that $X_t$ converges to $X$ in distribution, as $t \to \infty$. $N(\mu, \Sigma)$ denotes a normal random vector with mean vector $\mu$ and variance-covariance matrix $\Sigma$. For real sequences $\{a_t\}_{t\geq 0}$ and  $\{b_t\}_{t\geq 0}$, we write $a_t = \OO(b_t)$ if there exists two positive constants, say $C$ and $c$, such that $\limsup_{t \to \infty} | a_t/b_t |\leq C$ and $c\leq \liminf_{t \to \infty} | a_t/b_t |$. For a sequence of matrices $M(t) \in \IR^{k \times k}$ and a real-valued function $f$, we write $M(t) = \OO (f(t))$, whenever $M_{ij} (t) = \OO(f(t))$, for all $1 \leq i, j \leq k$. 
$\|\cdot\|$ denotes the standard $\LL^2$ norm and as defined before, $[N]$ denotes the set $\{1,\dots, N\}$. 
For a matrix $M$,  $\lambda_{\min}(M)$ and $\lambda_{\max}(M)$ denote the minimum  and the maximum  of the real part of the eigenvalues of $M$ respectively. For an eigenvalue $\lambda$, $\Re(\lambda)$ denotes the real part of $\lambda$. 


\section{Model dynamics} \label{Sec:prelim}
Let $G=(V, E)$ be a directed graph with $V = [N]$ and $E\subseteq V\times V$. For every vertex $i\in V$, let $V(i) \coloneqq  \{j \in V: (j,i)\in E\}$ denote the in-neighbourhood of $i$ and  define $\diin\coloneqq |\{ j\in V: (j,i)\in E\}| = |V(i)|$  and $\diout \coloneqq | \{j\in V: (i,j)\in E \}|$ to be  the in-degree and out-degree of  $i$ respectively. We assume that there are no isolated vertices in the graph, that is there does not exist any $i\in V$ such that $\diin = \diout =0$. We write $i\to j$, if there is a directed edge from $i$ to $j$ and  $i \rightsquigarrow j$, if there exists a path  $i=i_0\to i_1\to \dots \to i_{k-1}\to i_k=j$ from $i$ to $j$, for some $i_1,\dots, i_{k-1}\in V$. Throughout this paper we use the term \emph{d-regular graph} for a directed graph such that $\diin= \diout= d$ for every $i\in V$.

Suppose there is an urn at every vertex of $G$ such that each urn contains balls of two colours, white and black. Let $(W^t_i, B^t_i)$ be the configuration of the urn at vertex $i$, where $W^t_i$ and $B^t_i$ denote the number of white balls and black balls respectively, at time $t\geq 0$. Let $m \in \IZ^+ $ and $\alpha, \beta \in \{0, 1, \dots, m\}$ be fixed, then  given $\{(W^t_i, B^t_i)\}_{i\in V}$, we update the configuration  of each urn at time $t+1$ as follows:
	\begin{quote}
A ball is selected uniformly at random from all the urns simultaneously and independently of every other urn.  The colours of  these balls are noted and they are replaced into their respective urns. For every $i \in V$, if the colour of the ball selected  from the $i^{\mathrm{th}}$ urn is white, $\alpha$ white and $(m-\alpha)$ black balls are added to each urn $j$, such that $i\to j$; and if the colour of the ball selected from the $i^{\mathrm{th}} $ urn is black then $(m-\beta)$ white balls and $\beta$ black balls are added to each urn $j$, such that $i\to j$. 
\end{quote}
That is, each urn $i$ reinforces urn $j$ such that $i\to j$, according to the following reinforcement matrix
\begin{equation} \label{Rmatrix} R = \begin{pmatrix}
    \alpha & m-\alpha \\
    m-\beta & \beta
   \end{pmatrix}. 
   \end{equation}
    We call this a P\'olya type reinforcement if $\alpha= \beta= m$,  that is when $R=mI$ and non-P\'olya type reinforcement when $0<\alpha+\beta <2m$. To simplify the notation, we define $a \coloneqq \alpha/m$ and $b \coloneqq \beta/m$. 
Let $Z_i^t \coloneqq \dfrac {W^t_i}{W^t_i+B^t_i}$  be the proportion of white balls in the $i^{\mathrm{th}}$ urn at time $t$ and $Z^t\coloneqq  (Z_1^t, \dots, Z_N^t)$. Define the $\sigma$-field $\FF_t = \sigma \left( Z^0, Z^1, \ldots, Z^t \right)$ and let $Y^t_i$ denote the indicator of the event that a white ball is drawn from the $i^{\mathrm{th}}$ urn at time $t\geq 1$. Then the distribution of $Y_i^{t}$, conditioned on $\FF_{t-1}$ is given by
\[ Y^t_i =  \begin{cases}
1 & \text{ with probability } Z_i^{t-1},\\
 0 & \text{ with probability } 1- Z_i^{t-1}.
\end{cases} \]
Note that, conditioned on $\FF_{t-1}$,  $\{Y_i^t\}_{i\in V}$ are independent random variables.  
The proposed model is a general two-colour interacting urn model, as defined in Section \ref{sec1.1}, where 
the dependency set of the $i^{\mathrm{th}}$ urn is $V(i)$ such that given $\FF_{t-1}$, 
$(I_{i,W}^t, I_{i,B}^t) =  \sum_{j\in V(i)} (\alpha, m-\alpha) Y_j^{t-1}  + (m-\beta, \beta)(1-Y_j^{t-1})$, for $t\geq 1$.

We divide the vertex set of the graph into two disjoint sets,  say $ V=S\cup F$, where $S = \{ i \in V: \diin=0\}$ and $F = \{ i \in V: \diin>0\}$. We call the vertices in $S$ as \emph{stubborn} vertices, as the urns at these vertices  are not reinforced by any other urn i.e. $V(i)=\emptyset$  and we call $F$ as the set of \emph{flexible} vertices. Clearly, for every $i\in S$ 
\[W_i^t = W_i^0 \qquad \text{and } \qquad B_i^t = B_i^0, \qquad  \forall \ t >0, \]
whereas for every  $i \in F$,  we have 
\begin{eqnarray}
 W_i^{t+1} &=& W_i^t + \alpha \sum_{j\in V(i)} Y_j^{t+1} + (m-\beta) \sum_{j\in V(i)} (1-Y_j^{t+1}) \nonumber    \\
 &=& W_i^t + m\sum_{j\in V(i)}  \big( 1-b + (a+b-1) Y_j^{t+1} \big) \nonumber \\
 &=& W_i^t + m(1-b) \diin+ m(a+b-1)  \sum_{j\in V(i)} Y_j^{t+1}, \qquad  \forall \ t >0, 
  \label{Recur:Wt}
\end{eqnarray}
and 
 \begin{align*}
 B_i^{t+1} = B_i^t + mb \diin- m(a+b-1)  \sum_{j\in V(i)} Y_j^{t+1}, \qquad  \forall \ t >0.
\end{align*}

\noindent 
Let $ T_i^t = W_i^t +B_i^t$ be the total number of balls in the $i^{\mathrm{th}}$ urn at time $t$.
 Then, for every $i\in V$ 
\begin{equation}
\label{Equ:TotalBalls}
 T_i^t = T_i^{t-1} +m\, \diin= T_i^0 + tm\,\diin.
 \end{equation}
Thus,  the total number of balls in every urn is deterministic and increases linearly with time $t$, at every flexible vertex. 
In the next section, we briefly discuss stochastic approximation theory, a technique which is used in the later sections to obtain  the asymptotic results for $Z^t$ .


\section{Stochastic approximation scheme} \label{Sec:SA} 
A stochastic approximation scheme refers to a $k$-dimensional recursion of following type
\begin{equation}\label{basicSA}
 x_{t+1} = x_t + \gamma_{t+1} \left( h(x_t) +  \Delta M_{t+1} +\epsilon_{t+1} \right), \qquad \forall t \geq 0
\end{equation}
where $h: \IR^k \to \IR^k$ is a Lipschitz function, $\{\Delta M_t\}_{t\geq 1}$ is a bounded square-integrable martingale difference sequence and $\{\gamma_t\}_{t\geq 1}$ are positive step-sizes satisfying conditions that ensure that $ \sum_{t\geq 1} \gamma_t$ diverges, but slowly. More precisely, we assume the following.
\begin{itemize}
\item[(i)] $\sum_{t\geq 1} \gamma_t = \infty$ and $\sum_{t\geq 1} \gamma_t^2 < \infty$.
\item[(ii)] There exists $ C>0$ such that  $\IE\left [\|\Delta M_{t+1}\|^2 \vert \GG_t\right] \leq C$, almost surely  $\forall t \geq 0$,  where $\GG_t = {\sigma (x_0, M_1, \ldots, M_t)}$.
\item[(iii)] $\sup_{t\geq 0} \|x_t\| < \infty$, almost surely.
\item [(iv)] $\{ \epsilon_t \}_{t\geq 1}$ is a  bounded sequence such that $\epsilon_t \to 0$, almost surely as $t\to \infty$.
\end{itemize}	 
Then, the theory of stochastic approximation says that the iterates of \eqref{basicSA} converge almost surely to the stable limit points of the solutions of  the O.D.E. given by $\dot{x}_t = h(x_t)$. For explicit results and details on a standard stochastic approximation scheme (with $\epsilon_t = 0 \ \forall t \geq 0$), we refer the readers to Lemma 1 and Theorem 2 from Chapter 2 of \cite{borkar}.  Further in Chapter 2, the author discusses various extensions of the standard stochastic approximation scheme including the stochastic approximation scheme of the form \eqref{basicSA}. In this section, we use the stochastic approximation theory to study the limiting behaviour of $Z^t$ and  write  a stochastic approximation scheme  of the form \eqref{basicSA} for $Z^t$. We first establish some useful notation.

Without loss of generality, we take $F = \{1, \dots, |F|\}$ and $S = \{|F|+1,\dots, N\}$. For a vector $g\in \IR^N$ and $B \subseteq \{1, \ldots, N\}$, let $g_B$ denote the $|B|$-dimensional vector obtained by restricting $g$ to $B$. Similarly, for a matrix $M\in \IR^{N \times N}$ and $B \subseteq \{1, \ldots, N\}$, let $M_B$ denote the $|B| \times |B|$ matrix obtained by restricting $M$ to the index set $B \times B$. By $\bone$ and $\bz$ we denote a matrix of appropriate dimension with all entries equal to $1$ and $0$ respectively. 
We write $Z^t = (Z_1^t, \dots, Z_N^ t)= (Z_F^t, Z_S^0)$ as a row-vector and define $N \times N$ matrices
\[ \bT^{t} \coloneqq \begin{pmatrix} T_1^t & 0 & \cdots& 0\\  0& T_2^t & \cdots& 0\\        0& \cdots & \ddots& \vdots\\ 0& 0& \cdots& T_N^t \\ \end{pmatrix}
       \qquad and \qquad
       \bD \coloneqq  \begin{pmatrix} d_1^{\mathrm{in}} & 0 & \cdots& 0\\  0& d_2^{\mathrm{in}} & \cdots& 0\\        0& \cdots & \ddots& \vdots\\ 0& 0& \cdots& d_N^{\mathrm{in}}\\ \end{pmatrix}. \]
Then 
\[ \bT^{t} = \begin{pmatrix} \bT_F^t &\bz \\
 \bz  & \bT_S^0 \\ \end{pmatrix}
        \qquad and \qquad
       \bD =\begin{pmatrix} \bD_F&\bz  \\
      \bz &\bz \end{pmatrix}, \]
where $\bT_F^t$ and $\bD_F$ are both $|F| \times |F|$ matrices. 
Note that,  for $i\in S$ there is no $j\in V$ such that $j\to i$, therefore we can decompose the adjacency matrix $A$ of the graph as follows
\[A=\begin{pmatrix}
A_F& \bz\\
A_{S,F}  &\bz
\end{pmatrix},\]
where $ A_F $ is a $ |F|\times |F|$ matrix and $A_S$ is a $|S|\times |F|$ matrix.  We define the scaled adjacency matrix
\[ \tA  \coloneqq A \begin{pmatrix} \bD_F^{-1}&\bz  \\
      \bz &\bz \end{pmatrix} = \begin{pmatrix} \tA_F &\bz \\ \tA_{S,F} &\bz  \end{pmatrix}, \]
where  the  $(i,j)^{\mathrm{th}}$ entry of $\tA$  is equal to $\frac{1}{\djin} \I_{i \to j}$, whenever $d_j^{in}>0$ and equal to $0$ otherwise.
We now write the evolution of $Z^t_i$,  for $i \in F$, as a stochastic approximation scheme. From \eqref{Recur:Wt} and \eqref{Equ:TotalBalls} we get
\begin{eqnarray*}
 Z_i^{t+1} & = & \frac{1}{T_i^{t+1}} W_i^{t+1} \\
 & = & \frac{T_i^t}{T_i^{t+1}} Z_i^{t} +  \frac{m(1-b) \diin}{T_i^{t+1}} + \frac{m(a+b-1) }{T_i^{t+1}}   \sum_{j\in V(i)}  Y_j^{t+1} \\
 & = &   \left(1- \frac{m\diin}{T_i^{t+1}} \right)Z_i^{t} +   \frac{m(1-b) \diin}{T_i^{t+1}} + \frac{m(a+b-1) }{T_i^{t+1}}   \sum_{j\in V(i)}  Y_j^{t+1}. 
  \end{eqnarray*} 
Define  $  \Delta M_i^{t+1} \coloneqq (a+b-1) \left( Y_i^{t+1} - \IE\left[Y_i^{t+1}\vert \FF_t \right] \right) $, then we can write 
 \begin{eqnarray}
 Z_i^{t+1} & = &  \left(1- \frac{m\diin}{T_i^{t+1}} \right)Z_i^{t}  +   \frac{m(1-b) \diin}{T_i^{t+1}}   +\frac{m(a+b-1)}{T_i^{t+1}} \sum_{j \in V(i)}   \IE\left[Y_j^{t+1}\vert \FF_t \right] +\frac{m}{T_i^{t+1}} \sum_{j \in V(i)} \Delta M_j^{t+1} \nonumber \\
& = &   Z_i^t + \frac{m}{T_i^{t+1}} \left( (-Z_i^{t}  + 1-b) \diin + (a+b-1)  \sum_{j \in V(i)}  Z_j^t +\sum_{j \in V(i)} \Delta M_j^{t+1}\right). \label{Z_t_recur} 
\end{eqnarray} 
 Let $\Delta M^t \coloneqq ( \Delta M_1^t, \dots, \Delta M_{N}^t )$ be the martingale difference vector in $\IR^N$. We can write the recursion from \eqref{Z_t_recur} in vector form as follows
\begin{align}
 Z^{t+1}_F &= Z^t_F + m \bigg(  \left(-Z^t_F + (1-b) \bone\right) \bD_F +    (a+b-1)  (Z^t  A)_F +  \left(\Delta M^{t+1}  A\right)_F   \bigg)  \left( \bT_F^{t+1}\right)^{-1} \nonumber  \\
 & = Z^t_F + m\left(  -Z^t_F  + (1-b) \bone +   (a+b-1)\left( Z^t  \tA\right)_F +  \left(\Delta M^{t+1}  \tA\right)_F   \right)  \bD_F \left( \bT_F^{t+1}\right)^{-1}\\
  & = Z^t_F +   \left( h( Z^t_F)+  \left( \Delta M^{t+1}  \tA\right)_F    \right)  (m\bD_F)  \left(\bT_F^{t+1}\right)^{-1},
\end{align}
where $h:[0,1]^{|F|} \to \mathbb{R}^{|F|}$ is such that
 \begin{equation} \label{hfunction} 
h(z) = -z + (1-b) \bone +  (a+b-1) \left((z , Z_S^0) \tA\right)_F.
\end{equation}
Here  $\left((z , Z_S^0) \tA\right)_F  = z\tA_F + Z_S^0 \tA_{S,F}$, for $z\in [0,1]^{|F|}$. The above recursion can be written as 
\begin{equation*}
 Z_F^{t+1} = Z_F^t + \frac{1}{t+1}  \left( h(Z_F^t)+  \left(\Delta M^{t+1}  \tA\right)_F   \right) +  \frac{1}{t+1} \epsilon_{t+1},
\end{equation*}
where
\[\epsilon_{t+1} = \left( h(Z_F^t)  +   \left(\Delta M^{t+1}  \tA\right)_F  \right) \left( m(t+1)\bD_F \left( \bT_F^{t+1}\right)^{-1}  -I\right). \]
Thus, we get a recursion of the form \eqref{basicSA} given by
\begin{equation}  \label{Recursion-Vector} 
Z_F^{t+1} = Z_F^t +\gamma_{t+1} h(Z_F^t)  +\gamma_{t+1} (\Delta M^{t+1}\tA)_F +\gamma_{t+1}  \epsilon _{t+1},
\end{equation}
where $\gamma_t = \dfrac{1}{t}$ and   $(\Delta M^{t+1}\tA)_F$ is a bounded martingale difference sequence.  Using \eqref{Equ:TotalBalls}, it can be easily  verified that $\epsilon_t \to 0$ as $t \to \infty$ and  $h$ is a Lipschitz function.

Thus, the recursion in \eqref{Recursion-Vector}  satisfies the required assumptions (i)-(iv). Now, from the stochastic approximation theory, we know that the limit points of the recursion in \eqref{Recursion-Vector} almost surely coincide with the asymptotically stable equilibria of the O.D.E. given by
\begin{equation}
\label{ODE}
\dot{z} = h(z).
\end{equation}
Therefore, one can analyse the limiting behaviour of the interacting urn process by analysing the zeroes of the $h$ function and the eigenvalues of its Jacobian. Further, we use stochastic approximation results form \cite{Zhang2016} to establish central limit theorems for $Z_F^t$. 

\section{Main results} \label{Sec:MainResults}
In this section, we state our main results. In Section~\ref{sec:sync} we state an almost sure convergence theorem for $Z_F^t$ and  show that synchronisation occurs under certain conditions on the initial configuration of the urns and on the structure of the underlying graph. Recall that, by synchronisation we mean that the proportion of while balls converges to the same limit for every urn. Further, in Section~\ref{Sec:CLT} we state the central limit theorems for $Z_F^t$  for a non-P\'olya type reinforcement. In Section~\ref{Sec:Results-polya}, we consider a special case, namely P\'olya type reinforcement on a regular graph and show that $Z_F^t$ converge to a random vector almost surely. In addition, the urns synchronise in the sense that $Z_i^t$ converge to the same random variable for every $i\in F$. 


\subsection{Convergence and synchronisation results} \label{sec:sync}
\begin{theorem}\label{Thm:SLLN-NonPolya}
Suppose
\[ z^* \coloneqq \left( (1-b) \bone+ (a+b-1)Z^0_S \tA_{S,F} \right) \left(I-(a+b-1)\tA_F \right)^{-1}. \]
Then $ Z_F^t \longrightarrow  z^*$ almost surely as  $t\to \infty$, if either of the following holds

\begin{enumerate}[(i)]
\item Reinforcement is of non-P\'olya type that is, $R\neq mI$ and $a+b>0$.
\item  $S\neq \emptyset$ such that  for every $f \in F$ there exists $s \in S$ such that $s \rightsquigarrow f$.\end{enumerate}
\end{theorem}

\begin{remark}
Theorem \ref{Thm:SLLN-NonPolya} does not address the case when $S = \emptyset$ and $|a+b-1|=1$, in this case
the corresponding stochastic approximation scheme as in \ref {Recursion-Vector} holds with
$$ h(z)= -z(I-(a+b-1)\tA_F) +\bone.$$
Clearly, $h(z)=\mathbf{0}$  has a  unique solution whenever $I-(a+b-1)\tA_F$ is invertible.
 In fact when $I-(a+b-1) \tA_F$ is not invertible,  $Z_F^t$ need not admit a deterministic limit.  However, we expect that on a strongly connected graph, $ Z^{t}$ converges almost surely to a random limit. In Section \ref{Sec:Results-polya} we study the asymptotic of $Z^t$, when $S = \emptyset$ and $R=mI$ that is, $a+b-1=1$, on a regular graph as a special case. We show that $Z_i^{t}$ converges almost surely to the same random limit for every $i \in [N]$. 

\end{remark}

One straightforward conclusion from the above theorem is that  $Z^t \to (z^*, Z_S^0)$ almost surely, as $t\to \infty$.  Further, for synchronisation to occur among  the set of flexible vertices $F$,  we expect the reinforcement at any two flexible vertices to be similar. This manifests itself as a condition that $Z_S^0 \tilde{A}_S$ is a constant vector and that for $i, j \in F, i \neq j$, $\frac{|V(i) \cap F|}{\diin} = \frac{|V(j) \cap F|}{\djin} $. More precisely, we have the following result.

\begin{corollary}[Synchronisation] Suppose $Z^0_S \tA_{S,F} = c_1\bone $ and $\bone \tA_F =c_2 \bone$,  for some constants $c_1, c_2 \in [0,1]$, then under the condition $(i)$ or $(ii)$ of Theorem \ref{Thm:SLLN-NonPolya} 
 \[Z_F^t  \xrightarrow{a.s.}   \dfrac{(1-b) + (a+b-1) c_1 }{1-(a+b-1)c_2}\bone.\]
In particular, when  reinforcement is of non-P\'olya type and $S=\emptyset$,  then as $t \to \infty$
 \[Z_F^t  \xrightarrow{a.s.}     \dfrac{1-b }{2-a-b}\bone.\]
\label{Cor1}
\end{corollary}

\subsection{Fluctuation results} \label{Sec:CLT}
We now state the fluctuation theorems for $Z_F^t$ around the almost sure limit $z^*$.  Define
\begin{equation}\label{Def:rho} 
H \coloneqq -\frac{\partial h}{\partial z}=I-(a+b-1)\tA_F \quad \text{and} \quad \rho \coloneqq \lambda_{\min}(H),
\end{equation}
where $I$ is a $|F|\times |F|$ identity matrix. 
 As we will see in the results below, the scaling for fluctuation theorems depends on $\rho$. In the case, when $0<\rho < 1/2$, there exist finitely  many complex random vectors $\xi_1,\dots, \xi_l$ such that  $(Z_F^t-z^*)$  scaled appropriately can be almost surely approximated by a weighted sum of  $\xi_1,\dots, \xi_l$. The scaling in this case depends explicitly on $\rho$. For details we refer the reader to Theorem 2.2 of \cite{Zhang2016}. In this paper, we only discuss the cases $\rho > 1/2$ and $\rho = 1/2$ and obtain Gaussian limits with appropriate scaling. 



\begin{theorem} \label{Thm: CLT-rho> 1/2}
Suppose $ Z_F^t \longrightarrow  z^*$ almost surely as  $t\to \infty$ and $\rho >1/2$, then 
 \begin{equation}
\sqrt{t} \left(Z_F^t - z^* \right) \; \xrightarrow{d} N\left(\bz, \Sigma \right) \quad \text{as } t \to \infty,
\end{equation}
with
\[\Sigma =  \int_0^\infty \left( e^{-\left (H - \frac{1}{2} I \right)u} \right)^\top (\tA^\top \Theta \tA)_F  e^{-\left (H - \frac{1}{2} I \right)u} du,\]
where  $H$ is as defined in \eqref{Def:rho}  and  $\Theta$ is a  $N\times N$ diagonal matrix such that
\[\Theta_{i,i} =\begin{cases}
(a+b-1)^2 z_i^* (1-z_i^*) & \text{ for } i \in F,\\
(a+b-1)^2 Z_i^0(1-Z_i^0) & \text{ for } i \in S.
\end{cases}  \]
\end{theorem}
 

\begin{corollary}\label{Cor:Sigma-rho>1/2}
Suppose the reinforcement is of non-P\'olya type and $\tA=\tA^\top$. Then  for $\rho >1/2$   Theorem \ref{Thm: CLT-rho> 1/2} holds with
\[\Sigma = C(a,b)\; \tA_F^2 \left( I- 2(a+b-1)\tA _F\right )^{-1},\]
where $C(a,b) \coloneqq  \dfrac{ (a+b-1)^2(1-a)(1-b)}{(2-a-b)^2}.$
\end{corollary}


\begin{theorem} \label{Thm: CLT-rho=1/2}
Suppose $ Z_F^t \longrightarrow  z^*$ almost surely as  $t\to \infty$ and $\rho=1/2$, then 
\begin{equation}
 \sqrt{\frac{t}{ \log(t)}} \left(Z_F^t - z^* \right) \; \xrightarrow{d} N\left(\bz, \tilde \Sigma \right) \quad \text{as } t \to \infty,
\end{equation}
with
\[\tilde \Sigma =  \lim_{t\to \infty} \frac{1}{\log t } \int_0 ^ {\log t} \left( e^{ -(H-1/2I)u}\right)^\top \tilde (\tA^\top \Theta \tA)_F \left( e^{ -(H-1/2I)u}\right) du, \]
for $H$  as defined in \eqref{Def:rho}  and  $\Theta$  as defined in Theorem \ref{Thm: CLT-rho> 1/2}.
\end{theorem}

	
\begin{corollary}\label{Cor:Sigma-rho=1/2}
Suppose the reinforcement is of non-P\'olya type and $\tA=\tA^\top$. Then  for $\rho =1/2$   Theorem \ref{Thm: CLT-rho> 1/2} holds with
 	\[\tilde \Sigma =\frac{C(a,b)}{N}J,\]
where $J$ is a $|F|\times |F|$ matrix with all elements equal to $1$ and $C(a,b)$ is as defined in Corollary \ref{Cor:Sigma-rho>1/2}.
\end{corollary}


\begin{remark}[Friedman type reinforcement] As an immediate consequence of Corollary \ref{Cor1}, we get that with $S= \emptyset$ for the Friedman type reinforcement that is when $a=b \in (0,1)$,
$Z_F^t \longrightarrow \frac{1}{2} \bone $ almost surely. 
Further, in this case $C(a,b)$ simplifies and Corollary \ref{Thm: CLT-rho> 1/2} and Corollary \ref{Thm: CLT-rho=1/2} hold with $C(a,b) = \left( a-\frac{1}{2} \right)^2 $.
\end{remark}


\subsection{Convergence results for P\'olya type reinforcement} \label{Sec:Results-polya}
In this section, we consider the P\'olya type reinforcement on a finite directed graph. Throughout this section, we make the following assumptions.
\begin{enumerate}
\item [\bf (A1) ]  The graph is  $d$-regular and  the scaled adjacency matrix $\tA = \frac{1}{d} A$ is diagonalisable. More precisely,  there exists an invertible  matrix $P$ such that  $\tA = P \Lambda P^{-1} $ for a diagonal matrix $ \Lambda$  and $\lambda_{\max}(\tA)= 1$. 
\item [\bf (A2) ]  The initial number of balls in each urn is  the same, that is $\bT^0 = T^0 I$, for some $ T^0 \geq 1$. 
\end{enumerate}
Note that, for a $d$-regular graph we have $S=\emptyset$ and $F= V$, therefore to simplify the notation, we remove the subscript $F$ throughout this section.
In this case, the associated O.D.E. obtained from \eqref{ODE} reduces to
$ \dot{z} = h(z) = -z( I- \tA).$
The equilibrium points of this O.D.E. are the left eigenvectors of $\tA$  corresponding to the eigenvalue $1$. Observe that these equilibrium points are not stable as one of the eigenvalues of $\frac{\partial h}{\partial z}$ is $0$.
Therefore this case requires a different approach. Specifically, we use martingale theory to obtain convergence results. 
 
 \noindent
Under assumptions {\bf (A1)} and {\bf (A2)}
\begin{equation} \label{A:total}
 \bT^t =  (T^0+ tmd) I \eqqcolon T^t I .
\end{equation}
We define
\begin{equation*}
 \bar{Z}^t \coloneqq  \frac{\sum_{j=1}^N W_j^t}{\sum_{j=1}^N T_j^t} = \frac{\sum_{i=1}^N W_i^t}{N T^t} = \frac{1}{N} \sum_{i=1}^N Z_i^t =  \frac{1}{N} Z^{t} \bone^\top .
\end{equation*}


\begin{theorem}\label{Thm:Polya-regular-ConvgRate}
Suppose the reinforcement scheme is of P\'olya type, that is $R = m I$ and the assumptions {\bf (A1)} and {\bf (A2)} hold. Then
\begin{equation}\label{Var:phi}
\Var(Z^{t}- \bar{Z}^t \bone) =
\begin{cases} 
\OO\left (t^{-1}\right ) & \text{ if } \lambda_{\min}(\tA) <1/2, \\ 
  \OO\left(\frac{\log(t)}{t}\right )  & \text{ if } \lambda_{\min}(\tA) = 1/2, \\ 
\OO\left ( t^{2 \Re(\lambda_{(2)}) -2}\right ) & \text{ if } \lambda_{\min}(\tA)>1/2,
\end{cases}
\end{equation}
where $\lambda_{(2)}$ is the eigenvalue of $\tA$ with second largest real part.
Moreover, $ Z^{t}- \bar{Z}^t \bone \longrightarrow  0$  almost surely,  as  $t\to \infty$. 
\end{theorem}

\begin{corollary}\label{Cor:Polya-regular-ConvgRate}
Let $\hZ^t = Z^t\tilde{A}$ be the vector of neighbourhood averages of the proportion of white balls in the urns. Then under the setting of Theorem \ref{Thm:Polya-regular-ConvgRate},  $\Var(\hat{Z}^{t}- \bar{Z}^t \bone)  \longrightarrow  \bf{0}_{N\times N} $, as $t \to \infty.$
\end{corollary}

 The following theorem establishes that for every $i \in V$, $Z^t_i$ converge almost surely to the same limiting random variable. 

\begin{theorem}[Synchronization]  \label{Thm: SLLN-Polya-regular}
Suppose the reinforcement scheme is of P\'olya type, that is $R = m I$ and the assumptions {\bf (A1)} and {\bf (A2)} hold. Then  there exists a finite random variable $Z^\infty$ such that as $t \to \infty.$
\begin{equation}
 Z^t  \xrightarrow{a.s.}   Z^\infty \bone.
\end{equation}
\end{theorem}


\section{Proofs of the main results} \label{Sec:proofs}
In this section, we prove all the results from Section~\ref{Sec:MainResults}. For proving results from Section~\ref{sec:sync}, we use stochastic approximation theory, which was discussed in Section~\ref{Sec:SA}. We start with the following lemma.


\begin{lemma}\label{Lem:invertibility}
Let $M$ be a $n \times n$  real non-negative matrix   such that each column sum of $M$ is less than or equal to $1$. Then $rM-I$ is invertible for all $ r\in \mathbb{R}$ whenever $|r| <1$.
\end{lemma}
The above lemma follows from the Perron-Frobenius Theorem, however for the sake of completeness we include a proof here.
\begin{proof}[Proof of Lemma \ref{Lem:invertibility}]
 Note that, the case $r=0$ is trivial. For $r\neq 0$, suppose the matrix $(r M-I)$ is not invertible. Then there exists a vector ${\bf w} = (w_1, \ldots, w_n) \in \mathbb{C}^n$ such that 
$\bw (r M-I) =\bz$, that is $ \bw M = \frac{1}{r} \bw, $
which implies that for every $j\in \{1,\dots, n\}$
\begin{equation} \label{Eq:leftev}
\frac{1}{|r|} | w_j| =\big|\sum_{k=1}^n w_k M_{k,j}  \big|  
\leq (\max_{k} |w_k|)  \sum_{k=1}^n M_{k,j}  \leq  \max_k |w_k|.
\end{equation}
The last step follows since each column sum of $M$ is less than or equal to $1$. Now if $j = argmax \{|w_k|; k =1,\ldots, n\}$, then from \eqref{Eq:leftev} we must have $1/|r| \leq 1$, which contradicts our assumption of $|r|<1$. 
\end{proof}


\begin{proof}[Proof of Theorem \ref{Thm:SLLN-NonPolya}]
As discussed in Section \ref{Sec:SA},  it is enough to study the corresponding O.D.E. in \eqref{ODE},  of the stochastic approximation scheme for $Z^t_F$ obtained in \eqref{Recursion-Vector}. We know that a point $z^* \in [0, 1]^{|F|}$ is an equilibrium point of the associated O.D.E. $\dot{z} = h(z),$ if $h(z^*) = 0$. For the $h$ function given in equation \eqref{hfunction}, we have $h(z^*) = 0$, if and only if 
\begin{equation}
\label{Equ:ODEequil}
z^* - (a+b-1)( z^*\tA_F  +Z^0_S \tA_{S,F})  = (1-b)\bone.
\end{equation}
Thus the unique equilibrium point is given by 
\[z^* = \left( (1-b)\bone  +(a+b-1) Z^0_S \tA_{S,F}\right) \left(I- (a+b - 1)\tA_F \right)^{-1},\]
whenever  the matrix $I- (a+b - 1)\tA_F $ is invertible. We now show that under the assumptions of the theorem, this is indeed true. 
\begin{enumerate}
\item[(i)]  When $R\neq mI$ and $a+b>0$ we have $|a+b-1|<1$. Further, since each column sum of $\tA_F$ is less than or equal to $1$, the invertibility of matrix  $I- (a+b - 1)\tA_F $  follows from Lemma \ref{Lem:invertibility}.
\item[(ii)] We now  show that if for every $f \in F$ there exists $s \in S$ such that $s \rightsquigarrow f$, then $I-(a+b-1) \tilde{A}_F$ is invertible. Observe that, the set of flexible vertices $F$ can be partitioned into sets $F_1, \dots, F_k$  such that the graph $G$ restricted to $F_i$ is strongly connected,  for every $F_i$. Then the scaled adjacency matrix can be written as an upper block triangular matrix as follows
\[ \tA_F =\begin{pmatrix}
A_{F_1} & A_{F_1,F_2}  & \dots & A_{F_1,F_k}\\
0& A_{F_2}  & \dots & A_{F_2,F_k}\\
\vdots &\vdots & \ddots & \vdots  \\
0&\dots & 0& A_{F_k} 
\end{pmatrix}   \bD_F^{-1}  
\eqqcolon  \begin{pmatrix}
\tA_{F_1} & \tA_{F_1,F_2} & \dots & \tA_{F_1,F_k}\\
0&\tA_{F_2} & \dots & \tA_{F_2,F_k}\\
\vdots &\vdots & \ddots &  \vdots \\
0&0&\dots & \tA_{F_k} 
\end{pmatrix}, \]
where $\tA_{F_i, F_j}$ is a $|F_i|\times |F_j|$ matrix and  $\tA_{F_i}=\tA_{F_i, F_i}$. Since each $F_1,\dots, F_k$ is strongly connected, $\tA_{F_i}$ is an irreducible matrix for every $1\leq i\leq k$. Now

\[I- (a+b-1) \tA_F = \begin{pmatrix}
I_1- (a+b-1)\tA_{F_1} & \tA_{F_1,F_2} & \dots & \tA_{F_1,F_k}\\
0&I_2- (a+b-1)\tA_{F_2} & \dots & \tA_{F_2,F_k}\\
\vdots &\vdots & \ddots &  \vdots \\
0&0&\dots & I_k- (a+b-1)\tA_{F_k} 
\end{pmatrix}, \]
where $I_j$ is the $|F_j| \times |F_j|$ identity matrix. 
By Schur's complement,  we know that $I- (a+b-1) \tA_F$ is invertible  whenever each block matrix $ I_1-(a+b-1)\tA_{F_1}, \dots, I_k-(a+b-1)\tA_{F_k}$ is invertible. Since every  $\tA_{F_j}$ is an irreducible and sub-stochastic matrix (sub-stochasticity follows from the hypothesis in part (ii) of the theorem), therefore  $-1<\lambda_{\min}(\tA_{F_j})\leq \lambda_{\max}( \tA_{F_j})<1$ for every $j$. Hence $ I_j-(a+b-1)\tA_{F_j}$ is invertible for every $1\leq j\leq k$. 
\end{enumerate}
Finally,  using the above arguments we can also conclude that all the eigenvalues of the Jacobian matrix $\frac{\partial h}{\partial z} = (a+b-1)\tA_F-I$ are negative in both the cases. Therefore $z^*$ is uniformly stable and this concludes the proof.
\end{proof}

\begin{proof}[Proof of Corollary \ref{Cor1}]
Under the assumptions, the unique solution of \eqref{Equ:ODEequil}  is $z^*=c\bone $, for 
\[c =\frac{(1-b) +(a+b-1)c_1}{1- (a+b-1)c_2}.\]
For non-P\'olya type reinforcement we have $0< a+b \neq 2$  and  for $S=\emptyset$ we get $c_1=0$ and $c_2 =1$. Thus the unique equilibrium point $z^*$ simplifies to $z^* = \dfrac{1-b}{2-a-b} \bone.$
\end{proof}
We now prove the scaling limit theorems using tools and results from \cite{Zhang2016}. 
Observe that for $\rho >0$ ( where $\rho$ is as defined in \eqref{Def:rho}), the assumptions 2.2 and 2.3 made in  \cite{Zhang2016} are satisfied.  That is  for some  $\delta >0$ 
\[ h(z)= h(z^*) + (z-z^*) Dh(z^*) + o\left( \|z-z^*\|^{1+\delta} \right) \]
and for every $\epsilon >0$
\[ \frac{1}{t} \IE\left[\| (\Delta M^{t+1}\tA)_F  \|^2 \mathbb{I}\{\| (\Delta M^{t+1}\tA)_F  \|\geq \epsilon \sqrt{t}\} | \FF_t \right] \xrightarrow{a.s.}   0.\]
Further, as we shall see in the proofs below that $ \lim_{t \to \infty} \IE\left[\left((\Delta M^{t+1} \tA)_F  \right)^\top \left((\Delta M^{t+1} \tA)_F \right) \Big\vert \FF_t\right ]$ exists. 

\

As mentioned before, the scaling for the limit theorems is given by the regimes of $\rho$, where $\rho$ is as defined in \eqref{Def:rho}. We remark that $\rho$ depends on the eigenvalues of $\tA_F$ as well as on $a+b-1$, which is in fact the non-principal eigenvalue of the scaled reinforcement matrix $\frac{1}{m} R$. In particular,   we have the following two cases  
\begin{enumerate} 
\item When $\mathbf{a+b-1 > 0}$,  we get  $\rho = 1-(a+b-1) \lambda_{\max}(\tA_F).$
Therefore
\[ \rho>1/2 \iff \lambda_{\max}(\tA_F)  \in \left[-1,  \frac{1}{2(a+b-1)}\right)  \qquad \text{and} \qquad  \rho=1/2 \iff \lambda_{\max}(\tA_F) = \frac{1}{2(a+b-1)}.  \]

\item When $\mathbf{a+b-1 <0}$,  we get $\rho = 1-(a+b-1) \lambda_{\min}(\tA_F).$
 Therefore 
\[ \rho>1/2 \iff  \lambda_{\min}(\tA_F) \in \left(\frac{1}{2(a+b-1)} , 1\right] \qquad \text{and} \qquad \rho=1/2 \iff \lambda_{\min}(\tA_F) = \frac{1}{2(a+b-1)}.  \]
\end{enumerate}
We now give the proof of the limit theorems.

\begin{proof}[Proof of Theorem \ref{Thm: CLT-rho> 1/2}] 
The  limit theorem in case of $\rho>1/2$ holds with scaling $\sqrt{t}$ (see Theorem 2.3  \cite{Zhang2016}), 
with the limiting variance matrix given by
\begin{equation*}\label{Def:Sigma1}
\Sigma=\int_0^\infty \left(e^{-\left (H - 1/2I \right)u} \right)^\top\Gamma \, e^{-\left (H - 1/2I \right)u}  du, 
	\end{equation*}	
where
\[ \Gamma  = \lim_{t \to \infty} \IE\left[(\Delta M^{t+1} \tA)_F   ^\top \left((\Delta M^{t+1} \tA)_F  \right) \vert \FF_t\right ].\]
Note that, 
\[\left((\Delta M^{t+1} \tA)_F   \right)^\top \left((\Delta M^{t+1} \tA)_F  \right)  =  \left(  \tA^\top  \left(\Delta M^{t+1}\right)^\top \Delta M^{t+1} \tA\right)_F.\]
Now using the fact that for $i\neq j$, $Y_i^t$ and $Y_j^t$ are conditionally independent, we get
\[ \Gamma  =(a+b-1)^2 \, \left( \tA^\top  \lim_{t \to \infty}  \IE\left[( Y^{t+1}- Z^t)^\top (Y^{t+1} -Z^t)\vert \FF_t\right ]\tA\right) _F =  (\tA^\top \Theta \tA)_F \]
where $\Theta$ is a  diagonal matrix given by
\[\Theta_{i,i} =\begin{cases}
(a+b-1)^2 z_i^* (1-z_i^*) & \text{ for } i \in F,\\
(a+b-1)^2 Z_i^0(1-Z_i^0) & \text{ for } i \in S.
\end{cases}  \]
Therefore we get 
\[\Sigma  = \int_0^\infty \left(e^{-\left (H - \frac{1}{2} I \right)u} \right)^\top(\tA^\top \Theta \tA)_F\,  e^{-\left (H - \frac{1}{2} I \right)u} du.  \]
\end{proof}


\begin{proof}[Proof of Corollary \ref{Cor:Sigma-rho>1/2}]
For $\tA = \tA^\top$, we have  $S =\emptyset$  and $\tA_F = \tA$.  Therefore for non-P\'olya type reinforcement by Corollary \ref{Cor1}, $ z^* =\frac{1-b}{2-a-b} \bone$, thus $\Theta= C(a,b)I$ and we get 
\[  \Sigma  = C(a,b)\int_0^\infty \left( e^{-\left (H - \frac{1}{2} I \right)u} \right)^\top \tA^\top\tA \, e^{-\left (H - \frac{1}{2} I \right)u} du, \]
where $C(a,b) = \dfrac{ (a+b-1)^2(1-a)(1-b)}{(2-a-b)^2}.$
Further, $\tA$ has a spectral decomposition 
\begin{equation}
 \tA = P \Lambda P^{-1} \qquad  \text{ with } \qquad \Lambda = \begin{pmatrix} 
1 & 0& \cdots &0 \\
0&\lambda_1 & \cdots &0\\
\vdots & \vdots & \ddots &\vdots \\
0&0 &\cdots & \lambda_{N-1} \end{pmatrix},
\label{Equ:A-SpecDecomp}
\end{equation}
where $P$ is a $N \times N$ real orthogonal matrix and  $1, \lambda_1, \dots, \lambda_{N-1}$ are  eigenvalues of $\tA$. 
Since  $H$ commutes with $\tA$, the limiting variance matrix is given by 
\begin{align*}
\Sigma & = C(a,b)\; \tA^2 \int_0^\infty e^{-\left (2H - I \right)u} du \\
& = C(a,b)\; \tA^2 \int_0^\infty e^{-\left (I - 2(a+b-1)\tA \right)u} du \\
& =  C(a,b)\; \tA^2 P \left (\int_0^\infty e^{-\left (I - 2(a+b-1)\Lambda \right)u} du \right) P^{-1} \\
& =  C(a,b)\; \tA^2 \left(I- 2(a+b-1)\tA \right )^{-1}.
\end{align*}
This completes the proof.
\end{proof}


\begin{proof}[Proof of Theorem \ref{Thm: CLT-rho=1/2}]
In the case when $\rho =1/2$, the asymptotic normality holds with scaling $\sqrt{\dfrac{t}{\log t}}$ \ (see Theorem 2.1 \cite{Zhang2016}) and the limiting variance matrix given by
\[ \tilde \Sigma=\lim_{t \to \infty} \frac{1}{\log t}\int_0^{\log t} \left( e^{-\left (H - 1/2I \right)u} \right)^\top\Gamma e^{-\left (H - 1/2I \right)u} du, \]
where $\Gamma = (\tA^\top \Theta \tA)_F $
\end{proof}

\begin{proof}[Proof of Corollary \ref{Cor:Sigma-rho=1/2}]
For $\tA = \tA^\top$, as in the proof of Corollary \ref{Cor:Sigma-rho>1/2}, we get that $\Theta = C(a,b)I$  and $\tA_F =\tA$ is a symmetric matrix. Therefore in the case of $\rho=1/2$,  using the spectral decomposition as in  \eqref{Equ:A-SpecDecomp} the limiting variance matrix is given by
\begin{align*}
\tilde \Sigma & =  C(a,b)\; \tA^2 \lim_{t \to \infty} \frac{1}{\log t} \int_0^{\log t} e^{-\left (2H - I \right)u} du \\
& = C(a,b)\; \tA^2 \lim_{t \to \infty}  \frac{1}{\log t} \int_0^{\log t} e^{-\left (I - 2(a+b-1)\tA \right)u} du \\
& =  C(a,b)\; \tA^2 P \lim_{t \to \infty}  \frac{1}{\log t} \left (\int_0^{\log t} e^{-\left (I - 2(a+b-1)\Lambda \right)u} du \right) P^{-1},
\end{align*}
where
\[ e^{-(I - 2(a+b-1)\Lambda)u} = \begin{pmatrix} 
1 & 0& \cdots &0 \\
0&e^{-(1-2(a+b-1)\lambda_1)u} & \cdots &0\\
\vdots & \vdots & \ddots &\vdots \\
0&0 &\cdots & e^{-(1-2(a+b-1)\lambda_{N-1})u} \end{pmatrix}. \]
Since for $\rho =1/2$, $ 1-2(a+b-1)\lambda_i  >0$ for every  $i=1,\dots, N-1$,  we get 
\[\int_0^{\log t} e^{-(1-2(a+b-1)\lambda_i)u} du = \frac{1 - t^{-1+2(a+b-1) \lambda_i }}{1-2(a+b-1)\lambda_i}. \]
Thus, 
\[ \lim_{t \to \infty}  \frac{1}{\log t} \int_0^{\log t} e^{-(I-2(a+b-1)\Lambda)u} du = \begin{pmatrix} 
1 & 0& \cdots &0 \\
0& 0 & \cdots &0\\
\vdots & \vdots & \ddots &\vdots \\
0&0 &\cdots & 0 \end{pmatrix} \]
and we get 
\[\tilde{\Sigma} =  C(a,b)\; \tA^2 P \begin{pmatrix} 
1 & 0& \cdots &0 \\
0& 0 & \cdots &0\\
\vdots & \vdots & \ddots &\vdots \\
0&0 &\cdots & 0 \end{pmatrix} P^{-1}.\]
Since the normalized eigenvector corresponding to the maximal eigenvalue $1$ of $\tA$ is $\frac{1}{\sqrt{N}}\bone$ we get 
\[\tilde{\Sigma} =  C(a,b)\; \tA^2 \left (\frac{1}{N} J\right) = \frac{C(a,b)}{N}J, \]
where $J= \bone ^\top \bone $, is a $N \times N$ matrix with all elements equal to $1$.
\end{proof}
We now prove the results for P\'olya type reinforcement stated in Section \ref{Sec:Results-polya}. Our main aim  (as stated in Theorem \ref{Thm: SLLN-Polya-regular}) is to show that the process $\{Z_i^t\}_{t\geq 0}$ converges almost surely to a finite random variable $Z^\infty$ for every $i\in [N]$. In order to show this, we first prove that $Z^t-\bar{Z}^t$ converges almost surely to $\bz$, using the convergence result for almost supermartingales \cite{Robbins1971}. Then in the proofs of Lemma \ref{Quasi-Martingale} and  Theorem \ref{Thm: SLLN-Polya-regular}, we show that both $Z^t$ and $\bar{Z}^t$ admit an almost sure limit, using convergence theorems for quasi-martingales \cite{Michel1982} and martingales respectively.\begin{proof}[Proof of Theorem \ref{Thm:Polya-regular-ConvgRate}] 
Recall that under the assumption {\bf (A2)} of the theorem,  $ \bT^t = (T^0 + tmd) I = T^t I$. Define $\Phi_i^t = Z_i^t-\bar{Z}^t$, then in the vector notation we have
\begin{equation}\label{Defn:xt}
\Phi^t = Z^t - Z^t \frac{1}{N} J \eqqcolon  Z^t K,
\end{equation}
where $\Phi^t=\left(\Phi^t_1, \ldots, \Phi^t_N\right )$, $K =  I-\dfrac{1}{N}J $, and  $J$ is the $N\times N$ matrix of all entries equal to $1$. We want to show that the order of $\Var(\Phi^t) = \Var(Z^{t}- \bar{Z}^t \bone) $ is as given in equation \eqref{Var:phi} and $\Phi^t  \longrightarrow 0 $ almost surely.  To this end, we write a  recursion for $\Var \left(\Phi^{t+1}\right) $ using the law of total variance given by
\begin{equation} \label{total:var}
\Var \left(\Phi^{t+1}\right)  =   \Var \left(\IE\left[\Phi^{t+1}\Big\vert\FF_t \right] \right) + \IE \left[\Var \left(\Phi^{t+1}\Big\vert\FF_t \right) \right].  
\end{equation}
From equation \eqref{Recur:Wt} we get
\begin{align}
 \IE \left[Z^{t+1}\Big\vert \FF_t \right]
&= \frac{1}{T^{t+1}}  \IE\left[W^{t+1}\Big\vert\FF_t \right] = \frac{1}{T^{t+1}}    \left(W^{t} + m Z^tA \right) \eqqcolon Z^tU_t,
\label{Equ:Mean-Zt}
\end{align} 
where $U_t =  \dfrac{1}{T^{t+1}} (T^t I +m A )$. This gives 
 \begin{align}
\IE \left[\Phi^{t+1}\Big\vert\FF_t \right]
 &= Z^{t}U_t K  = \Phi^t U_t, \label{Exp-phi}
\end{align}
and
\begin{equation}\label{Eqn:Var-Exp-1}
\Var \left(\mathrm{E}\left[\Phi^{t+1}\Big\vert\FF_t \right] \right)  = U_t^\top \Var(\Phi^{t})U_t.
\end{equation}
For the second term of the expression for  $\Var \left(\Phi^{t+1}\right)$ in equation \eqref{total:var}, we have 
\begin{align*}
\Var\left(\Phi^{t+1}\Big\vert\FF_t \right)
&= K\Var\left(Z^{t+1}\Big\vert\FF_t \right) K \\ 
&= \frac{m^2}{(T^{t+1})^2} K  A^\top \, \Var\left(Y^{t+1}\Big\vert\FF_t \right)A K\\
&= \frac{m^2}{(T^{t+1})^2}K   A^\top \, \mathrm{Diag}\left(Z_i^{t} (1-Z_i^t)\right)_{1\leq i\leq N}    A K,
\end{align*}
where $\mathrm{Diag}\left( x_i \right)_{1\leq i\leq N} $ denotes the  $N\times N$ diagonal matrix with $i^{\mathrm{th}}$ diagonal element equal to $x_i$. 
Thus
\begin{equation}\label{Eqn:Exp-Var-2}
\IE \left[\Var \left(\Phi^{t+1}\Big\vert\FF_t \right) \right] 
 = \frac{m^2}{(T^{t+1})^2} K A^\top  V_z^t  AK,
\end{equation}
where $V_z^t = \mathrm{Diag}\big(\IE[Z_i^{t} (1-Z_i^t)]\,\big)_{1\leq i\leq N} $.
Substituting the quantities  from equations \eqref{Eqn:Var-Exp-1} and \eqref{Eqn:Exp-Var-2} in equation \eqref{total:var}, we get
\[ \Var \left(\Phi^{t+1}\right)   = U_t^\top \Var(\Phi^{t}) U_t+  \frac{m^2}{(T^{t+1})^2}  KA^\top V_z^t AK. \]
Iterating this we get
\begin{align} 
\Var \left(\Phi^{t+1}\right) 
&=  \left(\prod_{k=0}^{t-1} U_{t-k}^\top \right)  \Var \left(\Phi^{1}\right) \left(\prod_{k=1}^t U_k\right)  +  \sum_{j=1}^t  \frac{m^2}{(T^{j+1})^2}  \left(\prod_{k=0}^{t-j-1} U_{t-k}^\top \right)  K A^\top V_z^ j A K  \left(\prod_{k=j+1}^t U_k \right) \nonumber \\
&= \sum_{j=0}^t  \frac{m^2}{(T^{j+1})^2}  \left(\prod_{k=0}^{t-j-1} U_{t-k}^\top \right) KA^\top V_z^ j A K   \left(\prod_{k=j+1}^t U_k \right),\label{Var-1}
\end{align}
where the last equality follows from the fact that $\Var \left(\Phi^1\right)  = \frac{m^2}{(T^1)^2}KA^\top V_z^0 AK$.
We can write $U_t =  I+  \frac{md}{T^{t+1}}  \left (\tA-I\right ),$
where $\tA = \frac{1}{d} A$ is the scaled adjacency matrix. Using assumption {\bf (A1)},  we get  that there exists a matrix $P$ such that 
\[\tA = P \begin{pmatrix}
1 & 0 & \cdots & 0\\
0&\lambda_1 & \cdots & 0\\
\vdots&\vdots&\ddots&\vdots \\
0& 0 & \cdots & \lambda_{N-1}\\
\end{pmatrix} P^{-1} \qquad \text{and}  \qquad K= P \begin{pmatrix}
0 & 0 & \cdots & 0\\
0&1 & \cdots & 0\\
\vdots&\vdots&\ddots&\vdots \\
0& 0 & \cdots & 1\\
\end{pmatrix} P^{-1},\]
where $1, \lambda_1,\cdots, \lambda_{N-1}$ are the $N$ eigenvalues of $\tA$. Thus for $j \geq 0.$
\begin{equation}\label{Prod:U}
K \bigg(\prod_{k=j+1}^t U_k \bigg)  =P \prod_{k=j+1}^t \left ( I+ \frac{md}{T^0 +md(k+1) }   \left (\Lambda-I\right ) \right )P^{-1}.
\end{equation}
Let  $S \coloneqq  P^\top A^\top V_z^ j AP$, then from \eqref{Var-1} and \eqref{Prod:U} we get
\begin{align*}
&\Var \left(\Phi^{t+1}\right) \\
& =  (P^{-1})^\top \left( \sum_{j=0}^t  \frac{m^2}{(T^{j+1})^2}  \prod_{k=j+1}^t \left ( I+ \frac{md}{T^0 +md(k+1) }   \left (\Lambda-I\right ) \right ) S \prod_{k=j+1}^t \left ( I+ \frac{md}{T^0 +md(k+1) }   \left (\Lambda-I\right ) \right )\right) P^{-1} \\
  & = (P^{-1})^\top   \Omega P^{-1}.
\end{align*}
 where $\Omega  = \sum_{j=0}^t  \frac{m^2}{(T^{j+1})^2}  \prod_{k=j+1}^t \left ( I+ \frac{md}{T^0 +md(k+1) }   \left (\Lambda-I\right ) \right ) S \prod_{k=j+1}^t \left ( I+ \frac{md}{T^0 +md(k+1) }   \left (\Lambda-I\right ) \right )$.  The $(l,m)$-th element of $\Omega$ is 
\begin{align*}
\Omega _{l,m}  = \sum_{j=0}^t  \frac{m^2}{(T^{j+1})^2}   S_{l,m} \prod_{k=j+1}^t  \left (1+ \frac{\lambda_l -1}{(T^0/md) +k+1}\right ) \prod_{k=j+1}^t  \left (1+ \frac{\lambda_m -1}{(T^0/md) +k+1}\right ).
\end{align*}
Using Euler's approximation, for every $l =1, \dots N-1$, we get  $\prod_{k=j+1}^t  \left (1+ \frac{\lambda_l -1}{(T^0/md) +k+1}\right ) =\OO\left( (t/j)^{\lambda_l-1} \right)$. 
Since  the elements of matrix  $S$ are bounded,  we get 
\[ \Omega_{l,m}=  \OO\left( \sum_{j=0}^t  \frac{1}{j^2}   \left( \frac{t}{j}\right)^{\lambda_l+\lambda_m -2} \right) = \OO\left( \sum_{j=0}^t  \left( \frac{t}{j}\right)^{\lambda_l+\lambda_m } \right). \]
Hence
\[\Omega = \begin{cases} 
\OO\left(t^{-1}\right) & \text{ if } \lambda_{\min}(\tA) <1/2, \\ 
\OO\left (\frac{\log(t)}{t}\right ) & \text{ if } \lambda_{\min}(\tA) = 1/2,\\ 
\OO\left ( t^{2\Re( \lambda_{(2)}) -2}\right ) & \text{ if } \lambda_{\min}(\tA)>1/2.
\end{cases}\]
Since $P$ and $P^{-1}$ have bounded elements, the order for $ \Var \left(\Phi^{t+1}\right)$ is same as the order of $\Omega$.
This proves the first part of the theorem. Now from equation \eqref{Exp-phi} we get 
\[ \IE[\Phi^t] = \Phi^0 \prod_{j=1}^t U_j 
 =  Z^0 P \begin{pmatrix}
0 & 0 & \cdots & 0\\
0&\OO\left(t^{\lambda_1-1} \right)& \cdots & 0\\
\vdots&\vdots&\ddots&\vdots \\
0& 0 & \cdots & \OO\left( t^{\lambda_{N-1}-1}\right)
\end{pmatrix} P^{-1} \longrightarrow 0, \text { as } t \to \infty.\]
This implies  $\IE[ \|\Phi^t\|^2]  = \IE[ \Phi^t (\Phi^t)^\top] = \sum_{i=1}^N  \Var(\Phi^t_i) + \IE[\Phi^t] (\IE[\Phi^t])^\top \to 0$. Therefore, to show that $\Phi^t$ converges almost surely to $0$,  it is enough to show that $\|\Phi^t\|^2 $ admits an almost sure limit.  
 \begin{align*}
 \Phi^{t+1}  &= \IE \left[\Phi^{t+1}\Big\vert\FF_t \right]+   \Phi^{t+1} - \IE \left[\Phi^{t+1}\Big\vert\FF_t \right]\\
&= \IE \left[\Phi^{t+1}\Big\vert\FF_t \right]+   \frac{1}{(T^{t+1})^2}  \Delta M_\Phi^{t+1},
 \end{align*}
 where $\Delta M_\Phi^{t+1}=  \left(Y^{t+1} - \IE \left[Y^{t+1}\Big\vert\FF_t \right]\right) AK$ is a martingale difference sequence. Therefore we have 
 \begin{align}
\IE\left[ \|\Phi^{t+1}\|^2\Big\vert\FF_t \right] 
&= \IE \left[\Phi^{t+1} (\Phi^{t+1})^\top\Big\vert\FF_t  \right] \nonumber \\ 
& = \IE \left[\Phi^{t+1}\Big\vert\FF_t \right]  \IE \left[\Phi^{t+1}\Big\vert\FF_t \right] ^\top  +  \frac{1}{(T^{t+1})^2}  \IE \left[\Delta M_\Phi^{t+1} (\Delta M_\Phi^{t+1})^\top \Big\vert\FF_t \right] \nonumber \\ 
& = \Phi^{t} U_t U_t^\top  (\Phi^{t})^\top  +  \frac{1}{(T^{t+1})^2}  \IE \left[\Delta M_\Phi^{t+1} (\Delta M_\Phi^{t+1})^\top\Big\vert\FF_t \right] \nonumber \\ 
& = \Phi^{t}  \left(I-  \frac{md}{T^{t+1}}  \left (I- \tA \right ) \right)\left(I-  \frac{md}{T^{t+1}}  \left (I- \tA^\top  \right ) \right) (\Phi^{t})^\top  +  \frac{1}{(T^{t+1})^2}  \IE \left[\Delta M_\Phi^{t+1} (\Delta M_\Phi^{t+1})^\top\Big\vert\FF_t \right] \nonumber \\ 
& = \|\Phi^{t}\|^2 -  \frac{md}{T^{t+1}}   \Phi^{t} (2I-\tA-\tA^\top )   (\Phi^{t})^\top   +   \frac{m^2d^2}{(T^{t+1})^2}   \Phi^{t}  (I-\tA)(I-\tA^\top)     (\Phi^{t})^\top   \nonumber \\
&\quad +   \frac{1}{(T^{t+1})^2}   \IE \left[\Delta M_\Phi^{t+1} (\Delta M_\Phi^{t+1})^\top\Big\vert\FF_t \right]  \nonumber \\ 
& \leq  \|\Phi^{t}\|^2  +   \frac{\eta_t}{(T^{t+1})^2} , \label{norm:Phi}
 \end{align}
 where $ \eta_t = m^2d^2  \Phi^{t}  (I-\tA)(I-\tA^\top)   (\Phi^{t})^\top   +   \IE \left[\Delta M_\Phi^{t+1} (\Delta M_\Phi^{t+1})^\top\Big\vert\FF_t \right]  $ and  the last inequality follows from the fact that $ 2I-\tA-\tA^\top$ is a positive semi-definite matrix. Since $\eta_t$ is a bounded random variable, from equation \eqref{norm:Phi} we get that $\|\Phi^t\|^2$ is an almost super martingale, therefore  $\|\Phi^t\|^2 $ converges almost surely (see Theorem 1 \cite{Robbins1971}). This concludes the proof of the theorem. 
\end{proof}

\begin{proof}[Proof of Corollary~\ref{Cor:Polya-regular-ConvgRate}]
The proof follows from the same argument as above by taking $\Psi^t = Z^t \left (\tilde{A} - \frac{1}{N} J\right )$. \end{proof}


Recall that Theorem~\ref{Thm: SLLN-Polya-regular} says that for every $i \in V$, $Z_i^t$ converges to the same random limit. To prove this, we use the fact that $Z_i^t- \bar{Z}^t$ converges to $0$  almost surely (Theorem~\ref{Thm:Polya-regular-ConvgRate} proved above) and we show that both $Z^t$ and $\bar Z^t$ admit an almost sure limit.
From equation \eqref{Equ:Mean-Zt}, we have 
\[ \IE[Z^{t+1} | \FF_t] = Z^t \bigg( \frac{T^tI+ md\tA}{T^{t+1}}\bigg) = Z^t \bigg(I + \frac{md(\tA-I)}{T^{t+1}}\bigg).\]
Therefore, the $N$- dimensional process $\{Z^{t}\}_{t\geq 0}$ is a martingale if and only if $\tilde{A} = I$, that is when each node is isolated and has a self-loop. While $\{Z^{t}\}_{t\geq 0}$ is not a martingale in general, in the next lemma we show that $Z^t$ admits an almost sure limit. 

\begin{lemma} \label{Quasi-Martingale}
Under assumptions {\bf (A1)} and {\bf (A2)}  $Z^t$ admits an almost sure limit.
\end{lemma}
\begin{proof}
Using Theorem 11.7 in \cite{Michel1982},  it is enough to show that  process $(Z^t)_{t\geq 0}$ satisfies the following two conditions
\begin{enumerate}[(i)]
\item $\sum_{t=0}^\infty \IE\left [\| \IE[Z^{t+1} | \FF_t]-Z_t \|\right ]  <\infty$.
\item  $\sup_{t\geq 0}  \IE [\|Z^t\|] <\infty.$
\end{enumerate}
 Note that, $Z^t$ satisfies (ii) trivially. Let $\Phi_t$ and $\Psi_t$ be as defined in the proofs of Theorem~\ref{Thm:Polya-regular-ConvgRate} and Corollary~\ref{Cor:Polya-regular-ConvgRate} respectively.  
Now,
\begin{align}
\sum_{t=0}^\infty \IE\left [\| \IE[Z^{t+1} | \FF_t]-Z_t \|\right ] 
 =&\sum_{t=0}^\infty \frac{md}{T^{t+1}} \IE\left [\|Z^t(\tA-I) \|\right ] \nonumber \\
 = &\sum_{t=0}^\infty \frac{md}{T^{t+1}}  \IE\left [\big\|Z^t\left (\tA-\frac{1}{N}J\right ) + Z^t \left (\frac{1}{N}J - I\right ) \big\|\right ] \nonumber \\
 = &\sum_{t=0}^\infty \frac{md}{T^{t+1}}  \IE\left [\|\Psi^t - \Phi^t \|\right ] \nonumber \\
 \leq &\sum_{t=0}^\infty \frac{md}{T^{t+1}}    \bigg (\IE\left [\|\Psi^t -\IE[\Psi^t] \| \right ]+ \IE\left [\|\Phi^t - \IE[\Phi^t]\|\right ] + \IE\left [\|\IE[\Psi^t - \Phi^t]\|\right ] \bigg )\nonumber \\
 \leq & \sum_{t=0}^\infty \frac{md}{T^{t+1}}  \left (\sqrt{ \sum_{i=1}^N \Var(\Psi_i^t)} +\sqrt{ \sum_{i=1}^N \Var(\Phi_i^t)}+ \IE\left [\| \IE[\Psi^t - \Phi^t] \|\right ]\right ), \label{QuasiMar:Bound} 
\end{align} 
where the last inequality follows by Jensen's inequality. The fact that first two terms of \eqref{QuasiMar:Bound} are finite follows from Theorem \ref{Thm:Polya-regular-ConvgRate} and Corollary~\ref{Cor:Polya-regular-ConvgRate}. We will now show that the last term in \eqref{QuasiMar:Bound} is also finite.  We have
\[ \IE[\Phi^{t+1}]= \Phi^0 \prod_{j=1}^{t+1} \left(I + \frac{m}{T^j}   (A-dI )  \right) \qquad \text{and}\qquad  \IE[\Psi^{t+1}]= \Psi^0 \prod_{j=1}^{t+1} \left(I + \frac{m}{T^j} (A-dI )  \right). \]
Using assumption {\bf (A1)} and {\bf (A2)}  we get
\begin{eqnarray*}
\IE[\Psi^t - \Phi^t] & = & (\Psi^0-\Phi^0)\prod_{j=1}^t \left(I +\frac{m}{T^j}  (A-dI ) \right) \\
& = & Z^0(\tA - I )\prod_{j=1}^t \left(I +\frac{md}{T^j}  (\tA-I ) \right)\\
&=&Z^0 P(\Lambda -I) \prod_{j=1}^t \left(I + \frac{md}{T^j}  ( \Lambda -I ) \right) P^{-1} \\
& = &Z^0  P Q^t P^{-1},
\end{eqnarray*}
where $Q^t$ is a $N \times N$ diagonal matrix with $Q^t(i,i) = (\lambda_i-1) \prod_{j=1}^t \left ( 1+\frac{md(\lambda_{i}-1)}{T^j} \right ) = \OO(t^{\lambda_i-1})$ and $Q^t(1,1)= 0$. 
This implies $ \| \IE[\Psi^t - \Phi^t]\| =  \OO\left (t^{\Re(\lambda_{(2)})-1}\right )$.  Thus the sum in \eqref{QuasiMar:Bound} is finite. This completes the proof.
\end{proof}

\begin{proof}[Proof of Theorem \ref{Thm: SLLN-Polya-regular}]
From equation \eqref{Equ:Mean-Zt} we have 
\begin{align*}
\IE[\bar Z^{t+1} \vert\FF_t ] & =\frac{1}{N}\IE\left[Z^{t+1} \bone^\top \Big\vert\FF_t \right]\\
&= \frac{1}{N}Z^t \left (I+\frac{md}{T^{t+1}} (\tA-I) \right ) \bone^\top \\
&= \frac{1}{N}Z^t\bone^\top  = \bar Z^t,
\end{align*}
since $\tA \bone^\top =\bone^\top$ for a regular graph.
Thus, $\bar{Z}^t$ is a bounded martingale, therefore by the Martingale Convergence Theorem, there exists a finite random variable $Z^\infty$ such that $\bar{Z}^t \to Z^\infty $ almost surely. Using Theorem~\ref{Thm:Polya-regular-ConvgRate} and Lemma~\ref{Quasi-Martingale} we conclude that $Z^t \to Z^\infty \bone $ almost surely.
\end{proof}

Observe that for P\'olya type reinforcement, convergence and synchronisation results are obtained only for regular graphs. This restriction was needed to obtain explicit expressions by using the symmetry and hence the spectral decomposition of the adjacency matrix. We believe that the above results  can be extended to more general graphs using similar ideas.



\section{Application to opinion dynamics on networks}\label{Sec:Application}
We briefly demonstrate how these models can be used for modelling opinion dynamics on networks. We consider urns at vertices of a directed graph and assume that urns represent individuals such that proportion of white balls (respectively black balls) in an urn quantifies positive (respectively negative) inclination of that individual on a fixed subject. The graph structure determines the interactions between the individuals. That is, if there is an edge from node $i$ to node $j$, individual/urn at node $i$ can influence individual/urn at $j$ via a chosen (but fixed) reinforcement matrix. We define the opinion of an individual $i$ at time $t$ by $O^t_i = \mbox{Sign }(Z^t_i-1/2)$, where for convenience we assume $\mbox{Sign} (0) = 1$. The asymptotic results for the urn process obtained in this paper can be used to study the evolution of opinions defined this way. Note that, this process of evolution of opinions is very different from the traditional voter model or its extensions. In our model, the opinion of an individual does not flip frequently and evolves slowly, as the inclinations change depending on neighbourhood interactions. For instance, at time $t$ consider an individual $i$ with opinion $0$ (and $Z_i^t << 1/2$) such that all of her in-neighbours have opinion $1$. After the reinforcement at time $t+1$, it is quite probable that the positive inclination $Z_i^{t+1}$ of the individual at $i$ gets closer to $1/2$ but the opinion $O_i^{t+1}$ may still remain $0$. Thus, the opinion evolution model based on the interacting urn process on a network discussed in this paper models a more realistic behaviour of people.

The convergence results for the urn process could be used to answer some interesting questions about opinion evolution on a network. Given the reinforcement matrix, we can determine whether the limiting opinion of the majority is $1$ or $0$. Put differently, we can find conditions on reinforcement matrix such that the limiting opinion of the majority is $1$ or $0$. To illustrate this we consider the following example. 

\begin{example} Consider a directed star-graph $G=(V, E)$ on $N$ vertices such that the center vertex is labelled $1$ and rest of the vertices are labelled $\{2, 3, \ldots, N \}$. Then, $1$ is the only flexible  vertex with in-degree $N-1$  and $S= \{2, 3, \ldots, N \}$. Suppose the reinforcement matrix is 
 $R = \begin{pmatrix} \alpha & m-\alpha \\ m-\beta & \beta \end{pmatrix}$ for $\alpha, \beta \in \{0,1,\dots, m\}$.
Then from Theorem~\ref{Thm:SLLN-NonPolya}  we get that as $t\to \infty$ 
\[Z^t_1 \longrightarrow z^*= (1-b) + (a+b-1) \bar{Z}_{S}^0  \qquad \text{ almost surely,} \]
where $ \bar{Z}_{S}^0 = \frac{1}{N-1} \sum\limits_{j=2}^{N} Z_j^0$.
Thus the asymptotic opinion  of vertex $1$, that is $\mbox{Sign }(z^*-1/2)$ depends on $a, b$ and the average initial inclinations of stubborn individuals. 
\end{example}

\section*{Acknowledgment}
The work of Gursharn Kaur was supported by NUS Research Grant R-155-000-198-114. The work of Neeraja Sahasrabudhe was supported in part by the DST-INSPIRE Faculty Fellowship from the Government of India and in part by CEFIPRA grant No. IFC/DST-Inria-2016-01/448 “Machine Learning for Network Analytics”. We would like to thank Antar Bandyopadhyay for introducing us to this problem and his insightful remarks. We would also like to thank the reviewers for their valuable comments, that led to significant improvement of this work.

\bibliographystyle{abbrv}
\bibliography{ref}

\end{document}